\documentclass[12pt]{amsart}

\usepackage{amssymb,amsmath,latexsym}
\usepackage{amsthm}
\usepackage{amsfonts}
\usepackage{enumerate}
\usepackage{booktabs}

\usepackage{hyperref}

\usepackage[absolute]{textpos}

\usepackage{fancyhdr}

\newtheorem{thm}{Theorem}[section]

\newtheorem{prop}{Proposition}[section]
\newtheorem{lemma}{Lemma}[section]
\newtheorem{rmk}{Remark}[section]

\newcommand{\rmnum}[1]{\romannumeral #1}
\newcommand{\Rmnum}[1]{ \uppercase\expandafter{\romannumeral  #1}}

\newcommand{\refto}[1]{(\ref{eq:#1})}

\numberwithin{equation}{section}

\newcommand{\mysectionname}{}
\newcommand{\newsection}[1]{\section{#1}\renewcommand{\mysectionname}{\uppercase{#1}}}

\renewcommand{\headrulewidth}{0pt}

\begin{document}
\title{On regularity for measures in multiplicative free convolution semigroups}
\author{Ping Zhong}
\address{Department of Mathematics, Rawles Hall, 831 East Third Street, Indiana University, Bloomington, Indiana 47405, U.S.A. }
\email{pzhong@indiana.edu}
\date{}
\begin{abstract}
Given a probability measure $\mu$ on the real line, there exists a semigroup $\mu_t$  with real parameter $t>1$
which interpolates the discrete semigroup of measures $\mu_n$ obtained by iterating its free convolution. It was
shown in \cite{[BB2004]} that it is impossible that $\mu_t$ have no mass in an interval whose endpoints are atoms.
We extend this result to semigroups related to multiplicative free convolution. The proofs use subordination results.
\end{abstract}

\maketitle
\newsection{Introduction}
For any two probability measures $\mu$ and $\nu$ on the real line, we denote their free convolution by $\mu \boxplus \nu$.
In \cite{[BV1995]} Bercovici and Voiculescu observed that for any compactly supported measure
$\mu$ on the real line, one can defined $\mu_t$ for some $t>T$, where $T$ is a positive number which depends on
$\mu$. The result was extended in \cite{[NS1996]} by Nica and Speicher, who constructed a continuous semigroup $\{\mu_t : t > 1\}$ which
interpolates the discrete semigroup $\mu_n$ defined by
\begin{equation}
\mu_n=\mu \boxplus \mu \boxplus \cdots \boxplus \mu, n=1,2,\cdots
\end{equation}
Later, in \cite{[BB2004]} Belinschi and Bercovici used analytic method to give another construction of $\mu_t$, they also studied
regularity of $\mu_t$. In particular,
they showed that, for any $t>1$, it is impossible that $\mu_t$ have no mass
in an interval whose endpoints are atoms (see \cite[Proposition 3.3]{[BB2004]}).

If $\mu$, $\nu$ are probability measures on either the positive real line $\mathbb{R}_+ =[0,\infty)$ or on the unit
circle $\mathbb{T}$, we denote by $\mu \boxtimes \nu$ their free multiplicative convolution. Similar to additive free convolution, in \cite{[BB2005]}
Belinschi and Bercovici constructed a semigroup $\mu_t$ with real parameter $t>1$, for any measure on the
positive real line and for some measures (see Remark \ref{rmk}) on the unit circle. In this article, we use the tools in \cite{[BB2005]}
to extend this regularity result to multiplicative free convolutions. We also follow the notations in that paper.

We remark that the same type of results were proved for free convolution of \, two nontrivial measures in \cite{[BW2008]} by
Bercovici and Wang. For an introduction to free convolutions, we refer to the classic book \cite{[VDN1992]}.

We prove the regularity result for multiplicative free convolution of measures on the positive real line in Section 2,
and the corresponding result for measures on the unit circle in Section 3.


\newsection{Multiplicative Free Convolution on $\mathbb{R}_+$}
Let $\mu$ be a probability measure on the positive real line which is not a point mass at zero. We denote
$\Omega=\mathbb{C} \backslash [0,+\infty)$, and we define

\begin{equation}\label{psimu}
\psi_{\mu}(z)=\int_{\mathbb{R}_+}\frac{zt}{1-zt} \, d\mu (t), \hspace{6pt} z \in \Omega,
\end{equation}
and
\begin{equation}\label{etamu}
\eta_\mu(z)=\frac{\psi_\mu(z)}{1+\psi_\mu(z)},\hspace{6pt}  z \in \Omega.
\end{equation}

From the definition above, if $\mu((a, b))=0$ for two positive numbers $a<b$, then $\psi_\mu$ and $\eta_\mu$
are meromorphic on $\Omega \cup (1/b, 1/a)$ and are real valued on $(1/b, 1/a)$. We have the following result \cite[Proposition 2.2]{[BB2005]}:
\begin{prop}\label{etamu}
Let $\eta$ : $\Omega \rightarrow \mathbb{C}\backslash \{0\}$ be an analytic function such that $\eta(\overline{z})=\overline{\eta(z)}$ for
all $z \in \Omega$. The following two conditions are equivalent.
\begin{enumerate}
 \item There exists a probability measure $\mu \neq \delta_0$ on the $[0, +\infty)$ such that $\eta=\eta_\mu$.
 \item $\eta(0-)=0$ and $\arg\eta(z) \in [\arg z, \pi)$ for all $z \in \mathbb{C}^+$.
\end{enumerate}
\end{prop}

The function $\Sigma_\mu(z)=\eta_\mu^{-1}(z)/z$ is well-defined in a neighborhood of some interval $(-\alpha, 0)$.
Given two measures $\mu$ and $\nu$ on the positive real line, a new measure $\mu \boxtimes \nu$ is defined such that
$\Sigma_{\mu \boxtimes \nu}(z)=\Sigma_\mu(z) \Sigma_\nu(z)$. The measure $\mu \boxtimes \nu$ is called
the free multiplicative convolution of $\mu$ and $\nu$. See \cite{[VDN1992]} for details.

One can define a discrete semigroup $\mu_n$ by $\mu_n=\mu^{\boxtimes n}=\mu \boxtimes \mu \boxtimes \cdots \boxtimes \mu$.
The construction of a semigroup $\mu_t$ with real parameter $t > 1$ was given in \cite{[BB2005]}. The following theorem was proven
in that paper.

\begin{thm}\label{eq:thmdef}
Let $\mu \neq \delta_0$ be a probability measure on $[0, +\infty)$, and let $t>1$ be a real number.
\begin{enumerate}[(i)]
 \item There exists a probability measure $\mu_t \neq \delta_0$ on $[0, +\infty)$ such that $\Sigma_{\mu_t}(z)= \Sigma_\mu(z)^t$ for
$z<0$ sufficiently close to zero.
 \item There exists an analytic function $\omega_t : \Omega \rightarrow \Omega$ such that $\omega_t((-\infty, 0)) \subset (-\infty, 0)$,
$\omega_t(0-)=0$, $\arg\omega_t(z) \in [\arg z, \pi)$ for all $z \in \mathbb{C}^+$, and $\eta_{\mu_t}(z)=\eta_\mu(\omega_t(z))$ for all
$z\in \Omega$.
 \item The function $\omega_t$ is given by
 \begin{equation}\label{eq:defofomega}
 \omega_t(z)=\eta_{\mu_t}(z)\left[\frac{z}{\eta_{\mu_t(z)}} \right]^{1/t}, \hspace{6pt}  z \in \Omega,
 \end{equation}
where the power is taken to be positive for $z<0$.
 \item The analytic function $\Phi_t : \Omega \rightarrow \mathbb{C} \backslash \{0\}$ defined by
\begin{equation}\label{eq:defofphi}
\Phi_t(z)=z \left[\frac{z}{\eta_\mu(z)} \right]^{t-1}, \hspace{6pt}  z \in \Omega,
\end{equation}
satisfies $\Phi_t(\omega_t(z))=z$ for $z \in \Omega$.
\end{enumerate}
\end{thm}

The Cauchy transform of $\mu$ is defined by
\begin{equation}
G_\mu(z)=\int\frac{d \mu(t)}{z-t}.
\end{equation}
Given $\alpha \in \mathbb{R}$, then $(z-\alpha)G_\mu(z) \rightarrow \mu(\{\alpha\})$ as $z \rightarrow \alpha$
nontangentially to $\mathbb{R}$ (cf. \cite{[BV1998]}). We say $z \rightarrow \alpha$
nontangentially to $\mathbb{R}$ if $z$ approaches $\alpha$ and $| \Im{z}/ (\Re{z}-\alpha) |$ is bounded from below uniformly.
By definition of $\psi_\mu$ in the equation (\ref{psimu}), we have
$\psi_\mu(z)=-1+1/zG_\mu(1/z)$.
We can thus obtain the atoms of $\mu$ from $\eta_\mu(z)$ by its connection with the Cauchy transform of $\mu$.
A point $x \in (0, +\infty)$ is an atom for $\mu$ if and only if
$\eta_\mu(1/x)=1$ and $\eta_\mu'(1/x)=x/\mu(\{x\})$ is finite, where $\eta_\mu(1/x)$ is the limit of $\eta_\mu(z)$ when $z$ approaches $1/x$
from the upper half plane nontangentially and $\eta_\mu'(1/x)$ is the limit of $(\eta_\mu(z)-\eta_\mu(1/x))/(z-1/x)$
when $z$ approaches $1/x$ nontangentially, which is the
$Julia-Carath\acute{e}odory$ derivative of $\eta_\mu$ at $1/x$.

We need the following lemma which was presented in the proof of \cite[Proposition 5.2]{[BB2005]},
\begin{lemma}\label{eq:extension}
Given a probability measure $\mu$ on the positive real line, and using the notation in \refto{thmdef}, $\eta_{\mu_t}$ extends to
a continuous fucntion $\overline{\mathbb{C}^+} \backslash \{0\} \rightarrow \mathbb{C}$.
In particular, $\eta_{\mu_t}$ takes finite values on the interval $(0, +\infty)$.
\end{lemma}

Now we state our theorem, whose proof is standard with the help of the above results.
\begin{thm} \label{mainone}
Consider a probability measure $\mu\neq \delta_0$ on the positive real line, and let $t>1$.
If $\mu_t$ has atoms $a < b$, then we have $\mu_t((a, b)) >0$.
\end{thm}
\begin{proof}
We argue by contradiction. Suppose $\mu_t(\{a\})>0$,
$\mu_t(\{b\})>0$ and $\mu_t((a, b))=0$.
Let us first assume that $a>0$, in this case
$\eta_{\mu_t}(1/a)=\eta_{\mu_t}(1/b)=1$ and
$\eta_{\mu_t}$ is defined on $(1/b, 1/a)$, analytic and taking real
values. Notice that $\eta_{\mu_t}(\mathbb{C}^+) \subset
\mathbb{C}^+$, and $\eta_{\mu_t}(\overline{z}) =
\overline{\eta_{\mu_t(z)}}$, we claim that $\eta_{\mu_t}'(z)>0$ for
$z \in (1/b, 1/a)$. Indeed, for any $z \in (1/b, 1/a)$, $\eta_{\mu_t}(z)$ is real, thus
\begin{align}\eta_{\mu_t}'(z)&=\lim_{y \rightarrow 0}\frac{\eta_{\mu_t}(z+iy)-\eta_{\mu_t}(z)}{iy} \nonumber \\
                &=\lim_{y \rightarrow 0} \Re\{ \frac{\eta_{\mu_t}(z+iy)-\eta_{\mu_t}(z)}{iy}  \} \nonumber  \\
                &=\lim_{y \rightarrow 0} \frac{ \Im \eta_{\mu_t}(z+iy)}{y} \geq 0. \nonumber
\end{align}
If $\eta_{\mu_t}'(z_0)=0$ for some $z_0 \in (1/b, 1/a)$,
then the image under $\eta_{\mu_t}$
of a small disk $\{w : \Im w>0, |w-z|< \epsilon \}$ contains numbers in $\mathbb{C}^-$.
Thus, $\eta_{\mu_t}$ is increasing on $(1/b,
1/a)$. By Lemma \refto{extension}, $\eta_{\mu_t}$ can not be infinite, we have
$\eta_{\mu_t}\equiv 1$ on $(1/b, 1/a)$. This contradicts to
$\eta_{\mu_t}'(z)>0$ on $(1/b, 1/a)$.

Next we consider the case that $a=0$. In this case, $\eta_{\mu_t}(1/b)=1$ and $\eta_{\mu_t}$ is defined on
$(1/b, +\infty)$, analytic and taking real values. We also have $\eta_{\mu_t}'(z)>0$ for $z \in (1/b, +\infty)$.
By definition, we can calculate that $\lim_{z \rightarrow +\infty}\eta_{\mu_t}(z)=1-1/\mu_t(\{0\})<1$. This also contradicts to
the fact that $\eta_{\mu_t}$ is increasing and can not be infinite.
\end{proof}

\section{Free Multiplicative Convolution on $\mathbb{T}$}
Now we consider measures on the unit circle. We denote $\mathbb{D}=\{z : |z|<1\}$ and $\mathbb{T}=\{e^{it} | t\in [0, 2\pi) \}$.
We can now define the function $\psi_\mu$ and $\eta_\mu$
on the unit disk. Observe that
\begin{align}\label{psiMu}
\psi_\mu(z) &=\int_{\mathbb{T}}\frac{zt}{1-zt} \, d\mu(t) \nonumber \\
        &=\int_0^{2\pi} \frac{z}{e^{it}-z} \, d\mu(e^{-it}) \nonumber \\
        &=-\frac{1}{2}+\frac{1}{2} \int_0^{2\pi} \frac{e^{it}+z}{e^{it}-z} \,
        d\mu(e^{-it}), \hspace{6pt}  z\in \mathbb{D}.
\end{align}
Thus $\psi_\mu : \mathbb{D} \rightarrow \mathbb{C}$ is an analytic
function, and $\psi(0)=0$, $\Re\psi(z) \geq -\frac{1}{2}$
for all $z \in \mathbb{D}$.

Let us denote
$\eta_\mu=\psi_\mu/(1+\psi_\mu).$
It follows from the above obeservation that $\eta_\mu : \mathbb{D}  \rightarrow \mathbb{D}$,
$\eta_\mu(0)=0$ and $|\eta_\mu(z)| \leq |z|$. And it is well known that any analytic
function $\eta : \mathbb{D} \rightarrow \mathbb{C}$ such that
$|\eta_\mu(z)| \leq |z|$ for all $z \in \mathbb{D}$ is of the form $\eta_\mu$ for some probability measure $\mu$ on $\mathbb{T}$.

Suppose $\eta_\mu'(0)=\psi_\mu'(0)=\int_0^{2\pi}e^{it}\,
d\mu(e^{it}) \neq 0$, so that the inverse $\eta_\mu^{-1}$ is defined in a
neighborhood of zero. We denote
$\Sigma_\mu(z)=\eta_\mu^{-1}(z)/z$. Given two probability
measures $\mu$ and $\nu$ on $\mathbb{T}$, their free multiplicative
convolution, which is denoted by $\mu \boxtimes \nu$, is
characterized as in the case of measures on $[0, +\infty)$ by $\Sigma_{\mu \boxtimes
\nu}=\Sigma_\mu \Sigma_\nu$ in a neighborhood of zero.

Given $\delta_a$ for some $a\in \mathbb{T}$, one can easily check that $\Sigma_{\delta_a}(z)=1/a$. $\mu \boxtimes \delta_a$
is a probability measure on $\mathbb{T}$ such that $\mu \boxtimes \delta_a (at)=\mu(t)$ for $t \in \mathbb{T}$, i.e.
$\mu \boxtimes \delta_a $ can be obtained by rotating $\mu$ by $\arg a$.

The following theorem was proved in \cite{[BB2005]} (Theorem 3.5, Theorem 4.4 and Proposition 5.3).
\begin{thm}\label{eq:thmmulti}
Given a probability measure $\mu$ on $\mathbb{T}$ such that
$\int_{\mathbb{T}} \zeta \, d\mu(\zeta) \neq 0$, and the function
$\eta_\mu$ never vanishes on $\mathbb{D}\backslash \{0\}$. Let $t>1$
be a real number.
\begin{enumerate}[(i)]
 \item There exists a probability measure $\mu_t$ on $\mathbb{T}$ such
that $\int_{\mathbb{T}} \zeta \, d\mu_t(\zeta) \neq 0$, and
$\Sigma_{\mu_t}(z)=\Sigma_\mu(z)^t$ in a neighborhood of zero.
Moreover, $\eta_{\mu_t}$ never vanishes on $\mathbb{D}\backslash
\{0\}$.
 \item There exists an analytic function $\omega_t: \mathbb{D}
\rightarrow \mathbb{D}$ such that $|\omega_t(z)| \leq |z|$ and
$\eta_{\mu_t}(z)=\eta_\mu(\omega_t(z))$ for $z \in \mathbb{D}$.
 \item \label{omega} The function $\omega_t$ can be calculated as
$\omega_t(z)=\eta_{\mu_t}(z) \left[z/\eta_{\mu_t}(z)
\right]^{1/t}$, $z \in \mathbb{D}$.
 \item $\omega_t$ and $\eta_{\mu_t}$ can be extended as continuous functions from
$\overline{\mathbb{D}}$ to $\overline{\mathbb{D}}$. Moreover,
$\omega_t$ is one to one.
 \item If a point $\zeta \in \mathbb{T}$ and $\eta_{\mu_t}(\zeta)=1$, then
there is a real number $\theta$ such that $\zeta=e^{i\theta}$ and $\omega_t(\zeta)=e^{i\theta/t}$,
i.e. $\omega_t(\zeta)$ is one of the ($1/t$)-powers of $\zeta$.
\end{enumerate}
\end{thm}

\begin{rmk}\label{rmk}
\begin{enumerate}
 \item We need an additional assumption that the function
$\eta_\mu$ never vanishes on $\mathbb{D}\backslash \{0\}$ due to our
construction.
 \item We want to clarify that $\Sigma_{\mu_t}(0)=\lim_{z \to 0}z/\eta_{\mu_t}(z)$,
and the $1/t$ power in (\ref{omega}) above is chosen to be equal to $\Sigma_\mu(0)$ for $z=0$.
 \item The measures $\mu_t$ are only determined up to a rotation by a multiple of $2 \pi t$.
\end{enumerate}
\end{rmk}

By the above theorem and our discussion before, for $a \in \mathbb{T}$, we have $(\delta_a)_t=\delta_{a^t}$.
If $w=\mu \boxtimes \delta_a$, then $w_t=\mu_t \boxtimes \delta_{a^t}$ by choosing $w_t, \mu_t, a^t$ appropriately such
that $\Sigma_{w_t}=\Sigma_{\mu_t}\Sigma_{\delta_{a^t}}$.

Let $\zeta$ on the unit circle and $\alpha>1$, let
\begin{equation}\nonumber
\Gamma_{\alpha}(\zeta):=\{z\in \mathbb{D} : |z-\zeta|< \alpha(1-|z|)\}
\end{equation}
be a nontangential approach region ($Stoltz$ region). We say $z$ approaches $\zeta$ nontangentially if $z$ approaches $\zeta$
inside a nontangential approach region.
Similar to the discussion in the previous section, $1/\zeta \in \mathbb{T}$ is an atom of $\mu$ if and only if
$\eta_\mu(\zeta)=1$, where $\eta_\mu(\zeta)$ is the limit of $\eta_\mu(z)$ when $z$ approaches $\zeta$
nontangentially. And the $Julia-Carath\acute{e}odory$ derivative $\eta_\mu'(\zeta)$,
which is the limit of $(\eta_\mu(z)-\eta_\mu(\zeta))/(z-\zeta)$ when $z$ approaches $\zeta$ nontangentially,
is finite. In this case, we have
$\zeta\eta_\mu'(\zeta)=1/\mu( \{ 1/\zeta\} )$.

\begin{thm}
Let $\mu$ be a probability measure on $\mathbb{T}$ such that $\int_{\mathbb{T}} \zeta \, d\mu(\zeta) \neq 0$,
and the function $\eta_\mu$ never vanishes on $\mathbb{D}\backslash \{0\}$. Let $t>1$ be a real number.
Consider the measure $\mu_t$ constructed in (\refto{thmmulti}), and suppose that $\alpha$ and $\beta$ are atoms of $\mu_t$.
Then $\mu_t(I)>0$, where $I \subset \mathbb{T}$ is an open arc with endpoints $\alpha$ and $\beta$.
\end{thm}
\begin{proof}
We observe that for $t \geq2$, $\mu_t=\mu_{t/2} \boxtimes
\mu_{t/2}$, the result follows from the fact that a measure which is
of the form $\mu \boxtimes \mu$ can have at most one atom. This fact
is a direct consequence of \cite[Theorem 3.1]{[B2003]}. Therefore,
we only need to consider the case when $1 < t < 2$. Suppose there
exist two numbers $\alpha$, $\beta$ in $\mathbb{T}$ such that
$\mu_t(\{\alpha\})>0$, $\mu_t(\{\beta\})>0$ and $\mu_t(I)=0$. To
obtain a contradiction, we study the increment of argument of the
functions $\omega_t$ and $\eta_{\mu_t} \left[z/\eta_{\mu_t}
\right]^{1/t}$ when $z$ goes from $1/\alpha$ to $1/\beta$.

We denote $\overline{I} = \{ 1/{\zeta} | \zeta \in I \}$. Let us
assume $1/\alpha=e^{i\theta_1}$ and $1/\beta=e^{i\theta_2}$.
Replacing $\mu_t$ by $(\mu\boxtimes\delta_a)_t$ for some appropriate
$a$ if necessary, we may assume that $0< \theta_1 < \theta_2 < 2\pi$
and $\overline{I}=\{e^{it} : \theta_1 \leq t \leq \theta_2\}$.

Since $\mu_t(I)=0$, $\eta_{\mu_t}$ is analytic on $\overline{I}$.
Moreover, $|\eta_{\mu_t}(z)|=1$ for any $z \in \overline{I}$.
By the definition of $\eta_{\mu_t}$, one can easily check that $\eta_{\mu_t}(\mathbb{D}) \subset \mathbb{D}$ and
$\eta_{\mu_t}(\mathbb{C} \backslash \mathbb{D}) \subset \mathbb{C} \backslash \mathbb{D}$.
We claim that $\arg\eta_{\mu_t}(z)$ is monotonic in $\overline{I}$.
To see this, let us choose a conformal map $f$ from $\mathbb{D}$ to
$\mathbb{C}^+$ which transforms $\overline{I}$ to an interval $J
\subset \mathbb{R}$. Denote $\eta(z)=f \circ \eta_{\mu_t} \circ
f^{-1}$, then $\eta(\mathbb{C}^+)\subset\mathbb{C}^+$, and
$\eta(J)\subset \mathbb{R}$. By the proof of Theorem \ref{mainone},
$\eta'(z)>0$ for $z \in J$. This implies that $\arg\eta_{\mu_t}(z)$
is increasing on $\overline{I}$. More precisely, choose a continuous
function $g: [\theta_1, \theta_2] \rightarrow \mathbb{R}$ such that
$\eta_{\mu_t}(e^{i\theta})=\exp(ig(\theta)) \, (\theta_1 \leq \theta
\leq \theta_2)$. Then, $g$ is increasing. Similarly, we choose $h:
[\theta_1, \theta_2] \rightarrow \mathbb{R}$, such that
$\omega_t(e^{i\theta})=\exp(ih(\theta)) \, (\theta_1 \leq \theta
\leq \theta_2)$. Then $h$ is also increasing.

Notice that $\alpha$, $\beta$ are atoms of $\mu_t$, we have
$\eta_{\mu_t}(1/\alpha)=\eta_{\mu_t}(1/\beta)=1$.
$\mu_t(I)=0$, by the formula (\ref{psiMu}), $\psi_{\mu_t}$ is finite
on $\overline{I}$, therefore $\eta_{\mu_t}(\overline{I})\subseteq \mathbb{T}\backslash\{1\}$.
It implies that
the increment of argument of $\eta_{\mu_t}$ is $2\pi$
when $z$ goes from $1/\alpha$ to $1/\beta$.

By Theorem \refto{thmmulti}(\rmnum{5}), $\omega_t(1/\alpha)$(resp.
$\omega_t(1/\beta)$) is a ($1/t$)-power of $1/\alpha$
(resp.$1/\beta$),
there exist two integers $k_i (i=1, 2)$ such that $h(\theta_i)=1/t(\theta_i+2k_i \pi) (i=1, 2)$.
We obtain that
\begin{equation}\label{equa1}
h(\theta_2)-h(\theta_1)=1/t(\theta_2-\theta_1)+1/t(2(k_2-k_1)\pi).
\end{equation}
Also, there
exists an integer $k$ such that
\begin{equation}\label{overt}
\biggl(\frac{e^{i\theta}}{\eta_{\mu_t}(e^{i\theta})}
\biggr)^{1/t}=\exp\big(1/t(i\theta-g(\theta)+2k\pi)\big). \nonumber
\end{equation}
By the choices of $g$ and $h$, the above equation and $\omega_t=\eta_{\mu_t} \left[z/\eta_{\mu_t}
\right]^{1/t}$, we have
\begin{align}\label{equa2}
h(\theta_2)-h(\theta_1)&=(g(\theta_2)-g(\theta_1))+1/t[(\theta_2-\theta_1)-(g(\theta_2)-g(\theta_1))]\nonumber \\
     &=(1-1/t)(g(\theta_2)-g(\theta_1))+1/t(\theta_2-\theta_1) \nonumber \\
            &=(1-1/t)(2\pi)+1/t(\theta_2-\theta_1).
\end{align}
We compare the equation (\ref{equa1}) with the equation (\ref{equa2}),
and deduce that $t$ must be an integer. However, this is not true, since $1<t<2$.
\end{proof}

\section*{Acknowledgments}
The author thanks his advisor, Professor Hari Bercovici, for his generosity, encouragement and invaluable discussion during the course of the investigation.
He also thanks a referee for useful comments.


\begin{thebibliography}{99}
\bibitem{[B2003]}S.T.Belinschi, \emph{The atoms of the free multiplicative convolution of two probability distributions}, Integr. equ. oper. theory 46(2008), 377-386.
\bibitem{[BB2004]}S.T.Belinschi and H.Bercovici, \emph{Atoms and regularity for measures in a partial defined free convolution semigroup}, Math.Z., 248(2004), no. 4, 665-674.
\bibitem{[BB2005]}S.T.Belinschi and H.Bercovici, \emph{Partially defined semigroups related to multiplicative free convolution}, Int.Math.Res.Not. 2005, no.2, 65-101.
\bibitem{[BV1995]}H.Bercovici and D.Voiculescu, \emph{Superconvergence to the central limit and failure of the Cramer Theorem for free random Variables}, Probab.Theory Related Fields 103(1995), no.2, 215-222.
\bibitem{[BV1998]}H.Bercovici and D.Voiculescu, \emph{Regularity questions for free convolution}, Nonselfadjoint operator algebras, operator theory, and related topics, 37-47, Oper.Theory Adv.Appl., 104, Birkh\"{a}user, Basel,1998.
\bibitem{[BW2008]}H.Bercovici and J-C.Wang, \emph{On freely indecomposable measures}, Indiana Univ. Math.J., 57(2008), no.6, 2601-2610.
\bibitem{[NS1996]}A.Nica and R.Speicher, \emph{On the mulitiplication of free N-tuples of noncommutative random variables}, Amer.J.Math. 118(1996), no.4, 799-837.
\bibitem{[VDN1992]}D.Voiculescu, K.J.Dykema and A.Nica, \emph{Free Random Variables}, CRM Monograph Series, vol.1, AMS, Rhode Island, 1992.








\end{thebibliography}
\end{document}